\documentclass[12pt,epsfig,amsfonts]{amsart} 
\usepackage{amsmath,amsthm,amssymb,amscd,epsfig,color}
\usepackage{graphicx}
\usepackage{mathrsfs}

\newtheorem{theorem}{Theorem}

\newtheorem{lemma}[theorem]{Lemma}

\theoremstyle{definition}

\numberwithin{equation}{section}

\begin{document}

\author{Hiroki Takahasi}

\address{Keio Institute of Pure and Applied Sciences (KiPAS), Department of Mathematics,
Keio University, Yokohama,
223-8522, JAPAN} 
\email{hiroki@math.keio.ac.jp}
\urladdr{\texttt{http://www.math.keio.ac.jp/~hiroki/}}

\subjclass[2010]{11A55, 11K50, 37A45}
\thanks{{\it Keywords}: continued fractions, partial quotients, Hausdorff dimension}


\title[]
 {Hausdorff dimension of sets with\\ restricted, slowly growing partial quotients} 
 \maketitle
 
 \begin{abstract}
 I. J. Good (1941) showed that the set of irrational numbers in $(0,1)$ whose 
partial quotients $a_n$ 
 tend to infinity is of Hausdorff dimension $1/2$.
 A number of related results 
 impose restrictions of the type $a_n\in B$ or $a_n\geq f(n)$,
  where $B$ is an infinite subset of $\mathbb N$ and $f$ is a rapidly growing function with $n$.
We show that, for an arbitrary $B$ and an arbitrary $f$ with values in $[\min B,\infty)$ and tending to infinity, the set
of irrational numbers in $(0,1)$ such that
\[ a_n\in B,\ a_n\leq f(n)\text{ for 
all $n\in\mathbb N$, and }a_n\to\infty\text{ as }n\to\infty\]
is of Hausdorff dimension $\tau(B)/2,$
where $\tau(B)$ is the exponent of convergence of $B$. 
 \end{abstract}

 \section{Introduction}
 Let $\mathbb I$ denote the set of irrational numbers in $(0,1)$.
  Each $x\in\mathbb I$ has
 an infinite continued fraction expansion 
 $x=[a_1(x),a_2(x),\ldots]=
 1/(a_1(x)+1/(a_2(x)+\cdots))$, where
the positive integer 
$a_n(x)$ is called a {\it partial quotient of $x$}.
 Sets of irrational numbers whose partial quotients obey various conditions have been studied since the works of Jarn\'ik \cite{Jar28} and Good \cite{Goo41}. 
 Among a number of results Good obtained 
in \cite{Goo41}, 
the principal one states that
\[\label{good1}\dim_H\left\{x\in\mathbb I\colon \lim_{n\to\infty}a_n(x)=\infty\right\}=\frac{1}{2},\]
where $\dim_H$ denotes the Hausdorff dimension,
see \cite[Theorem~1]{Goo41}.
Some refinements and extensions of Good's result are available. 
Ramharter \cite{Ram85} showed that
the set of numbers in $\mathbb I$ 
with strictly monotone increasing partial quotients is of Hausdorff dimension $1/2.$
Jaerisch and Kesseb\"ohmer \cite{JaeKes10} considered the
set of numbers in $\mathbb I$ whose all partial quotients are greater than $q$, and
obtained a precise asymptotics of the Hausdorff dimension of this set as $q\to\infty$ 
using the thermodynamic formalism. Jordan and Rams \cite{JorRam12} extended Ramharter's result to some iterated function systems.
 
 In \cite{Hir73}, Hirst proved results analogous to that of Good concerning the cases where $a_n$
 is restricted to belong to some sequence of natural numbers.
 For an infinite subset $B$ of $\mathbb N=\{1,2,\ldots\}$,
 define
 \[E(B)=\left\{x\in\mathbb I\colon a_n(x)\in B\text{ for all } n\in\mathbb N,\text{ and } \lim_{n\to\infty}a_n(x)=\infty\right\},\]
and define the {\it exponent of convergence} of $B$ by
 \[\tau(B)=\inf\left\{s\geq0\colon\sum_{
 k\in B}k^{-s}<\infty\right\}.\]
 Hirst showed
 that $\dim_H E(B)\leq\tau(B)/2$
 (see \cite[Corollary~1]{Hir73}), and
 conjectured that the equality holds
 for an arbitrary $B$. He treated
 the special case $B=\{k^b\}_{k\in\mathbb N}$, $b$
 a positive integer
 to support his conjecture (see \cite[Theorem~3]{Hir73}).
 Cusick \cite[Theorem~1]{Cus90} proved that the equality holds in the case $B$ is not too sparse and satisfies what he called the density assumption.
 After 35 years of the appearance of Hirst's paper \cite{Hir73},
 Wang and Wu solved the conjecture of Hirst in the affirmative, 
 see  \cite[Theorem~1.1]{WanWu08'}.
 For an extension of Wang and Wu's result to iterated function systems, see \cite{CWW13}.


The Hausdorff dimension of sets of the form
 \[F(B,f)=\{x\in E(B)\colon a_n(x)\geq f(n)\text{ for all } n\in\mathbb N\},\]
 where $f$ is any function which tends to infinity with $n$ has also been considered in the literature.
  In \cite[p.227]{Hir73}, Hirst conjectured that the equality
$\dim_HF(B,f)=\tau(B)/2$ holds no matter how rapidly $f$ grows, and 
 Cusick later showed the case $B=\mathbb N$ and $f(n)=2^{2^{2^n}}$ 
as a counterexample (see \cite[Lemma~3]{Cus90}).
Nowadays it is known that the Hausdorff dimension of the set
\[\{x\in\mathbb I\colon a_n(x)\geq f(n)\text{ for all } n\in\mathbb N\}\]
decreases as $f$ grows more rapidly.
Good proved that
it is $1/2$ if $f(n)<(\log n)^A$
for some $A>0$, for all sufficiently large $n$
(see \cite[Theorem~3]{Goo41}).
Hirst improved the upper bound of Good on $f$
to certain double exponential functions
 (see \cite[Theorem~2]{Hir70}).
For an arbitrary double exponential function $f(n)=c^{b^n}$ where $b,c>1$,
$\L$uczak \cite{Luc97}, Feng et al. 
\cite{FWLT97}
showed that the Hausdorff dimension of the set is
 $\frac{1}{1+b}$, after a partial result of Moorthy \cite{Moo92}.
 A related result can be found in \cite{JorRam12}.
Another related result is due to Wang and Wu 
\cite[Theorems~3.1 and 4.2]{WanWu08} which
gives a complete description of the Hausdorff dimension of sets 
$\{x\in\mathbb I\colon a_n(x)\geq f(n)\   \text{for infinitey many} \ n\in\mathbb N\}$.
For an arbitrary infinite subset $B$ of $\mathbb N$ and an arbitrary double exponential function $f$,
Cao, Wang and Wu linked $\tau(B)$ to the Hausdorff dimension of $F(B,f)$ 
(see \cite[Theorem~1.2]{CWW13}).
In \cite[Lemma~3.2]{FLWJ09}, Fan et al. computed the Hausdorff dimension of sets of the form
\[\{x\in\mathbb I\colon s_n\leq a_n(x)<Ns_n\text{ for all }n\in\mathbb N\},\]
where $s_n\geq3$, $s_n\to\infty$ as $n\to\infty$ and $N\geq2$.

 In view of these developments, it is also relevant to describe the Hausdorff dimension of subsets of $E(B)$
 with slowly growing partial quotients, namely,
 sets of the form
 \[G(B,f)=\left\{x\in E(B)\colon a_n(x)\leq f(n)\text{ for all }n\in\mathbb N\right\},\]
 where $f$ is a function taking values in $[\min B,\infty)$ and tending to infinity with $n$.
  Our main result states that the Hausdorff dimension of this set never drops from $\tau(B)/2$ no matter how slowly $f$ grows.
 \begin{theorem}\label{main}
 For any infinite subset $B$ of $\mathbb N$ and
 any function
  $f\colon\mathbb N\to[\min B,\infty)$ such that $\lim_{n\to\infty}f(n)=\infty$,
   the set
   \[G(B,f)=\left\{x\in\mathbb I\colon a_n(x)\in B,\  a_n(x)\leq f(n)\ \text{\rm for all }n\in\mathbb N, \ \text{\rm and } \lim_{n\to\infty}a_n(x)=\infty\right\}\]
   is of Hausdorff dimension $\tau(B)/2.$
 \end{theorem}


 Since $E(B)$ contains $G(B,f)$,
 Theorem \ref{main} yields
 $\dim_HE(B)\geq\tau(B)/2$.
 From this and the upper bound
 $\dim_HE(B)\leq\tau(B)/2$
  obtained by Hirst \cite[Corollary~1]{Hir73},
  we obtain 
$\dim_HE(B)=\tau(B)/2,$
a positive answer
to the first conjecture of Hirst as proved by
Wang and Wu in \cite[Theorem~1.1]{WanWu08'}.
The original proof of Wang and Wu does not use ergodic theory.

By virtue of the upper bound obtained by Hirst \cite[Corollary~1]{Hir73},
for a proof of Theorem \ref{main}
it suffices to show the lower bound $\dim_H G(B,f)\geq\tau (B)/2$.
Our proof of this relies on the ergodic theory 
for the iteration of the Gauss map which generates the partial quotients in the continued fraction expansion.
 After introducing notations in Section~2, we construct in Section~3 
 a sequence of Bernoulli measures with non-uniform weights, supported on finitely many $1$-cylinders indexed by elements of $B$
 and having dimensions (the Kolmogorov-Sinai entropy divided by the
 Lyapunov exponent, see \cite[Theorem~4.4.2]{MauUrb03}) not much smaller than $\tau(B)/2$. 
 This construction relies on a novel use 
 of the exponent of convergence,
 as in \eqref{bn} below.
 Using these measures, 
 in Section~4 we construct fractal subsets
  of $G(B,f)$ and Borel probability measures supported on them, 
 and appeal to the mass distribution principle
 \cite{Fal} to estimate from below
 the Hausdorff dimension of these fractal sets.

 \section{Notations}
 Following the monograph of Khinchin \cite{Khi64}
 we introduce basic notations.
 For each $n\in\mathbb N$ and
 $(a_1,a_2,\ldots,a_n)\in\mathbb N^{n},$ we call the interval
 \[I(a_1,a_2,\ldots,a_n)=\begin{cases}
 \left[\frac{p_n}{q_n},\frac{p_{n-1}+p_n}{q_{n-1}+q_n}\right)& \text{ if $n$ is even,}\\
 \left(\frac{p_{n-1}+p_n}{q_{n-1}+q_n},\frac{p_n}{q_n}\right]& \text{ if $n$ is odd,}\end{cases}\]
 an {\it $n$-cylinder}. Here, $p_k$, $q_k$ $(k=0,\ldots, n)$ are given by the recursion formulas
 \[p_{-1}=1,\ p_0=0,\ p_k=a_kp_{k-1}+p_{k-2}\quad\text{for}\ \ k=1,\ldots,n,\]
 \[q_{-1}=0,\ q_0=1,\ q_k=a_kq_{k-1}+q_{k-2}
 \quad\text{for}\ \ k=1,\ldots,n.\]
 The $n$-cylinder $I(a_1,a_2,\ldots,a_n)$ represents
 the set of all real numbers in $[0,1)$ which have a (finite or infinite) continued fraction expansion beginning by $a_1,a_2,\ldots,a_n$, i.e.,
 \[I(a_1,a_2,\ldots,a_n)=\{x\in[0,1)\colon a_1(x)=a_1,\ a_2(x)=a_2,\ldots,a_n(x)=a_n\}.\]
In the case of $1$-cylinders, note that $I(k)=(\frac{1}{k+1},\frac{1}{k}]$ for each $k\in\mathbb N$.

 We denote 
 $(a_1,\ldots,a_n)\in\mathbb N^n$
 simply by $\omega$
 and write $I(\omega)$
 for the $n$-cylinder $I(a_1,\ldots,a_n).$
 For $\omega_1=(a_1,\ldots,a_n)\in\mathbb N^n$
 and $\omega_2=(b_1,\ldots,b_m)\in\mathbb N^m$, 
 we write $(\omega_1,\omega_2)$ for $(a_1,\ldots,a_n,b_1,\ldots,b_m)\in\mathbb N^{n+m}$
 and so on.

 \section{Measures supported on  restricted $1$-cylinders}
 The continued fraction expansion is generated by iterations of the Gauss map
 $T\colon(0,1]\to[0,1)$ given by
  \[T(x)= \frac{1}{x}-\left\lfloor\frac{1}{x}\right\rfloor,\]
 where $\lfloor y\rfloor$ denotes the largest integer not exceeding $y$.
For each $x\in\mathbb I$
 and $n\geq1$, $a_n(x)=\lfloor 1/T^{n-1}(x)\rfloor $ holds.
  We say a Borel probability measure $\mu$ on $[0,1]$ is {\it $T$-invariant} if
 $\mu(\mathbb I)=1$ and $\mu(T^{-1}(A))=A$ holds for any Borel subset $A$ of $\mathbb I$. 
We say a $T$-invariant Borel probability measure $\mu$ is {\it ergodic} if
$T^{-1}(A)=A$ for a Borel subset $A$
of $\mathbb I$ implies $\mu(A)=0$
or $=1$.
 For a $T$-invariant Borel probability measure $\mu$, let $h(\mu)$ denote the Kolmogorov-Sinai entropy of $\mu$
 relative to the restriction of $T$
 to $\mathbb I$. 
 Since the set of $1$-cylinders 
 generates the Borel sigma-algebra of $\mathbb I$,
 we have
\[h(\mu)=-\lim_{n\to\infty}\frac{1}{n}\sum_{(a_1,\ldots,a_n)\in\mathbb N^n}\mu(I(a_1,\ldots,a_n))\log\mu(I(a_1,\ldots,a_n))\in[0,\infty],\]
 where $0\log0:=0$.
 Define the
 {\it Lyapunov exponent} of $\mu$ relative to $T$ by
 \[\chi(\mu)=\int\log|T'|d\mu\in\left[2\log\frac{\sqrt{5}-1}{2},\infty\right].\]


 \begin{lemma}\label{measure}
 Let $B$ be an infinite subset of $\mathbb N$ with $\tau(B)>0$.
 For any $\epsilon\in(0,\tau(B)/2)$
  there exist a strictly increasing sequence 
  $\{b_m\}_{m\in\mathbb N}$
in $B$ and a sequence $\{\mu_m\}_{m\in\mathbb N}$ of $T$-invariant ergodic Borel probability measures on $[0,1]$
 with the following properties:
 
 \begin{itemize}
 \item[(a)] for all $m\in\mathbb N$, \[\sum_{\stackrel{k\in B}{b_m\leq k<b_{m+1}}}\mu_m(I(k))=1;\]  
\item[(b)] for all $m\in\mathbb N$, both $h(\mu_m)$ and $\chi(\mu_m)$ are finite and 
 \[\liminf_{m\to\infty}\frac{h(\mu_m)}{\chi(\mu_m)}\geq
 \frac{\tau(B)}{2}-\epsilon.\]

 \end{itemize}
 \end{lemma}
 \begin{proof}
 Let $\epsilon\in(0,\tau(B)/2)$.
 We define a strictly increasing sequence $\{b_m\}_{m\in\mathbb N}$ in $B$ inductively
 as follows: set $b_1=\min B$, and
 $b_{m+1}>b_m$ is the minimal integer in $B$ such that
  \begin{equation}\label{bn}
  \sum_{\stackrel{b_m\leq k< b_{m+1}}{k\in B}}k^{-\tau(B)+\epsilon}\geq 1.
 \end{equation}
  This definition makes sense since $\tau(B)$ is the exponent of convergence of $B$.
  Put
 $B_m=\{k\in B\colon b_m\leq k<b_{m+1}\}.$
 Let $\mu_m$ denote the Bernoulli measure which assigns to each $1$-cylinder
 $I(k)$, $k\in B_m$ the probability
 $Z_mk^{-\tau(B)+\epsilon},$
 with the normalizing constant $Z_m=\left(\sum_{k\in B_m}k^{-\tau(B)+\epsilon}\right)^{-1}.$
 Item (a) is obvious from the definition of $\mu_m$.
 
 Since
  $\mu_m$ is a Bernoulli measure, it is $T$-invariant and ergodic.
   The entropy of $\mu_m$ is given by
   \[h(\mu_m)=-\sum_{k\in B_m}\mu(I(k))\log \mu(I(k)).\]
Then 
\begin{equation}\label{entropy} \begin{split}
    h(\mu_m)
    &=-\sum_{k\in B_m}Z_mk^{-\tau(B)+\epsilon}\log (Z_mk^{-\tau(B)+\epsilon})\\&=-\log Z_m+(\tau(B)-\epsilon) Z_m\sum_{k\in B_m}k^{-\tau(B)+\epsilon}\log k\\
    &\geq (\tau(B)-\epsilon) Z_m\sum_{k\in B_m}
    k^{-\tau(B)+\epsilon}\log k,
    \end{split}\end{equation}
because of $Z_m\leq1$ from \eqref{bn}.
For the Lyapunov exponent of $\mu_m$ we have
\begin{equation}\label{exponent}
\chi(\mu_m)=2\int|\log x|d\mu_m(x)\leq 2Z_m\sum_{k\in B_m}
k^{-\tau(B)+\epsilon}\log(k+1).\end{equation}
From \eqref{entropy} and \eqref{exponent}, for all sufficiently large $m$ we have
\[\frac{h(\mu_m)}{\chi(\mu_m)}\geq
\frac{1}{2}(\tau(B)-\epsilon)\min_{k\in B_m}\frac{\log k}{\log(k+1)}
\geq\frac{1}{2}\tau(B)-\epsilon,\]
which implies (b).
\end{proof}

 \section{Proof of Theorem~\ref{main}.}
 Let $B$ be an infinite subset of $\mathbb N$
 and $f\colon\mathbb N\to[\min B,\infty)$ be a function such that $\lim_{n\to\infty}f(n)=\infty$.
 By virtue of the upper bound obtained by Hirst \cite[Corollary~1]{Hir73},
  it is enough to show the lower bound \begin{equation}\label{desire}\dim_H G(B,f)\geq\frac{1}{2}\tau (B).\end{equation}
    We assume $\tau(B)>0$, for otherwise there is nothing to prove.
    For any number strictly less than $\tau(B)/2$,
    we will construct a subset of $G(B,f)$ whose Hausdorff dimension exceeds that number.

  Let $\epsilon\in(0,\tau(B)/2).$
 Let $\{b_m\}_{m\in\mathbb N}$ be a strictly monotone increasing sequence in $B$ and let $\{\mu_m\}_{m\in\mathbb N}$ be a sequence of $T$-invariant ergodic Borel probability measures
  for which the conclusions of Lemma~\ref{measure} hold.
 Set
 \begin{equation}\label{def-bm}
 B_m=\begin{cases}\{\min B\}&\text{ for $m=1$.}\\
 \{k\in B\colon b_m\leq k<b_{m+1}\}&\text{ for $m\geq2$.}\end{cases}\end{equation}
  For each $m\geq1$ and an integer $\ell\geq1$ we define  \[B_m^{\ell}=\{(a_1,\ldots,a_{\ell})\in\mathbb N^{\ell}\colon a_i\in B_m\quad \text{for}\ \ i=1,\ldots, \ell\}.\]
Put $\ell_1=1$ and $A_1^{\ell_1}=B_1$ for ease of notation. Let $m\geq2$.
From Birkhoff's ergodic theorem and Shannon-McMillan-Breiman's theorem (see e.g. \cite{CFS}) applied to the measure $\mu_m$,
and the fact (see \cite[Chapter~7, $\S$4]{CFS}) that there exists a constant $C>1$ such that $|(T^n)'(x)|/|(T^n)'(y)|\leq C$ holds for every $n\geq1$
and all $x,y\in(0,1]$ contained in the same $n$-cylinder,
there exist an integer $\ell_m\geq1$ and a subset $A_m^{\ell_m}$
of $B_m^{\ell_m}$ such that
\begin{equation}\label{qe1}
 \left|\frac{1}{\ell_m}\log\#A_m^{\ell_m}- h(\mu_m)\right|\leq\frac{1}{m},\end{equation}
 and 
\begin{equation}\label{qe2}
\left|\frac{1}{\ell_m}\log |(T^{\ell_m})'(x)|- \chi(\mu_m)\right|\leq\frac{1}{m}\quad \text{ for any}\ \ x\in\bigcup_{\omega\in A_m^{\ell_m}} I(\omega),
\end{equation}
where the symbol $\#$ denotes the cardinality of sets.
Following the notation in the end of Section~2,
for each $m\geq1$ and $t\geq1$ put 
\[A_m^{t\ell_m }=\{(\omega_1,\ldots,\omega_{t})\in\mathbb N^{t\ell_m}\colon\omega_i\in A_m^{\ell_m}\quad
\text{for}\ \ i=1,\ldots,t\}.\]
 Let $\{t_m\}_{m\in\mathbb N}$ be a sequence of positive integers such that
  for every $m\geq2$ we have
  \begin{equation}\label{bm}b_{m+1}\leq\inf\left\{ f(n)\colon \sum_{j=1}^{m-1}t_j\ell_j+1\leq n\leq\sum_{j=1}^mt_j\ell_j\right\}.\end{equation}
  Since $\lim_{n\to\infty}f(n)=\infty$,
  one can choose such a sequence 
  by induction on $m$. 

We represent each integer $M\geq t_1$ as
\begin{equation}\label{k}
M=t_1+\cdots+t_{m}+s,\ 0\leq s\leq t_{m+1}-1,\end{equation}
and introduce
a finite subset of $\mathbb N^{t_1\ell_1+\cdots+t_m\ell_m+s\ell_{m+1}}$ as
   \[ \Omega(t_1,\ldots,t_{m},s)=\left\{\begin{tabular}{l}
 $ \left\{\begin{tabular}{l}$(\omega_1,\ldots,\omega_{m+1})\colon\omega_j\in A_j^{t_j\ell_j}\quad\text{for}\ \  j=1,\ldots,m$\\
and$\ \ \omega_{m+1}\in A_{m+1}^{s\ell_{m+1}}$\end{tabular}\right\}\quad$ if $s\neq0$,\\
 $\left\{(\omega_1,\ldots,\omega_m)\colon \omega_j\in A_j^{t_j\ell_j}\quad\text{for}\ \  j=1,\ldots, m\right\}\quad\quad\quad$ if  $s=0$\\
  \end{tabular}\right\}.\]
Set
\[\Lambda= \bigcap_{M=t_1}^{\infty}\overline
{\bigcup_{\omega\in \Omega(t_1,\ldots,t_m,s)} I(\omega)}.\]
The set $\Lambda$ is an intersection of decreasing compact sets, and so a non-empty compact set.
By construction we have \[E(B)\supset\Lambda\cap\mathbb I.\]
Let $x\in\Lambda\cap\mathbb I$.
By construction,
for each $m\geq1$ and $0\leq s\leq t_{m+1}-1$ there exists $\omega\in \Omega(t_1,\ldots,t_m,s)$ such that $x\in I(\omega).$ The first alternative of \eqref{def-bm} gives
 \[ a_n(x)=\min B\leq  f(n)\quad \text{for}\ \ n=1,\ldots,t_1\ell_1.\]
For every $m\geq2$, the second alternative of \eqref{def-bm} and \eqref{bm} yield
 \[a_n(x)< b_{m+1}\leq  f(n)\quad \text{for} \ \ n=\sum_{j=1}^{m-1}t_j\ell_j+1,\ldots,\sum_{j=1}^m
 t_j\ell_j.\]
 It follows that $x\in G(B,f)$, and therefore
 \[G(B,f)\supset\Lambda\cap\mathbb I.\]

For each $\omega\in  \Omega(t_1,\ldots,t_{m},s)$ fix
a point $x(\omega)\in I(\omega)\cap \Lambda$. Let $\nu_M$ denote the uniform probability distribution on the finite set $\bigcup\{x(\omega)\colon\omega\in  
\Omega(t_1,\ldots,t_{m},s) \},$
namely 
$\nu_M=(1/\#\Omega(t_1,\ldots,t_{m},s))
\sum_{\omega\in\Omega(t_1,\ldots,t_{m},s)}\delta_{x(\omega)}$
where $\delta_{x(\omega)}$
denotes the unit point mass at $x(\omega)$.
Let $\nu$ be an accumulation point of the sequence $\{\nu_M\}_{M\geq t_1}$
in the weak* topology on the space of Borel probability measures on $[0,1]$.
We claim that
$\nu(\Lambda)=1$. Indeed,
 since $\nu$
is regular (see \cite[Theorem~6.1]{Wal82}) and $\Lambda$ is compact, 
$\nu(\Lambda)<1$ would imply
the existence of an open set $C\subset[0,1]$ such that  $C\cap\Lambda=\emptyset$
and $\nu(C)>0$.
Since $\nu_M(C)\leq\nu_M([0,1]\setminus\Lambda)=0$ for all $M$, it would follow that
$\nu(C)\leq\liminf_{M\to\infty}\nu_M(C)=0$, a contradiction.
We show that
\begin{equation}\label{mass}
\liminf_{r\to 0}\frac{\log \nu (B(x,r)\cap\Lambda)}{\log r} 
\geq \liminf_{m\to\infty}\frac{h(\mu_m)}{\chi(\mu_m)}\quad\text{for all} \ \ x\in \Lambda,
\end{equation}
where $B(x,r)=(x-r,x+r).$ 
Then,
from the mass distribution principle \cite[p.60]{Fal}
and Lemma~\ref{measure}(b) it follows that
\[\dim_H\Lambda\geq\liminf_{m\to\infty}\frac{h(\mu_m)}{\chi(\mu_m)}\geq\frac{1}{2}\tau(B)-\epsilon.\]
Since $\epsilon\in(0,\tau(B)/2)$ is arbitrary, we obtain \eqref{desire}.

The rest of this paper is devoted to the proof of \eqref{mass}. 
For each integer $M\geq t_1$ in \eqref{k},
 we set \begin{equation}\label{choose-eq0}
r_{M}= e^{-t_{m}\ell_{m}\left(\chi(\mu_{m})+ \frac{3}{m}\right)
-s\ell_{m+1}\left(\chi(\mu_{m+1})+ \frac{2}{m+1}\right)}.\end{equation}
We choose $t_m$ $(m=1,2,\ldots)$ inductively 
so that $r_M$ is strictly monotone decreasing and converges to $0$ as $M\to\infty$.

Let $r\in(0,t_{1}]$, and let $M\geq t_1$ be such that
\begin{equation}\label{ak}r_{M+1}<r\leq r_M.\end{equation}
For each $x\in\Lambda$
we estimate $\nu(B(x,r)\cap\Lambda)$ from above.
Let $|J|$ denote the Euclidean length of a bounded interval $J$.
For each $\omega\in \Omega(t_1,\ldots,t_{m},s)$ we have
\[
\begin{split}
\sup_{I(\omega)}\log|(T^{\sum_{j=1}^mt_j\ell_j+s\ell_{m+1}})'|&\leq\sum_{j=1}^{m} t_{j}\ell_j\left(\chi(\mu_{j})+\frac{1}{j}\right)+s\ell_{m+1}\left(\chi(\mu_{m+1})+\frac{1}{m+1}\right)\\
&\leq t_{m}\ell_m\left(\chi(\mu_{m})+\frac{2}{m}\right)
+s\ell_{m+1}\left(\chi(\mu_{m+1})+\frac{1}{m+1}\right)\\
&\leq \log \frac{1}{r_M}.
\end{split}
\]
The first inequality follows from  \eqref{qe2}.
The second one holds provided
  $t_m$ is chosen to be large enough
compared to $t_1,\ldots,t_{m-1}$.
The last one is by \eqref{choose-eq0}.
Using
the mean value theorem and the above estimate of the derivative, 
\[
|I(\omega)|\geq \frac{|T^{\sum_{j=1}^m t_j\ell_j+s\ell_{m+1}}(I(\omega))|}{\sup_{I(\omega)}|(T^{\sum_{j=1}^m t_j\ell_j+s\ell_{m+1}})'|}=
 \frac{1}{\sup_{I(\omega)}|(T^{\sum_{j=1}^m t_j\ell_j+s\ell_{m+1}})'|}
\geq r_M.\]
Since $r\leq r_M$, this lower bound immediately gives an upper bound on the number 
of intervals $I(\omega)$ which intersect a fixed interval of length $2r$ intersecting $\Lambda$.
For any $x\in \Lambda$  we have 
 \begin{equation}
 \begin{split}\label{covering}
   \#\left\{\omega \in \Omega(t_1,\dots ,t_{m},s)   \colon B(x,r)\cap I(\omega)\neq \emptyset
   \right\}&\le \frac{2r_{M}}{\inf
_{\omega\in   \Omega(t_1,\ldots,t_{m},s)} |I(\omega)|}+2\\
&\leq 4.
 \end{split}
 \end{equation}
Let 
$M'=t_1+\cdots+t_{m'}+s'$ be an integer with $m'\geq m$, $0\leq s'\leq t_{m'+1}-1$ and $M'\geq M+2$.
If
$\omega\in \Omega(t_1,\dots ,t_{m},s)$ and 
$\omega'\in \Omega(t_1,\dots ,t_{m'},s')$, then
$\overline{I(\omega)}\cap \overline{I(\omega')}=\emptyset$,
or $\overline{I(\omega')}$ is contained in the interior of $I(\omega)$.
In particular,
for any $\omega\in \Omega(t_1,\dots ,t_{m},s)$,
$\partial I(\omega)\cap\Lambda=\emptyset$ holds.
Since $\nu(\Lambda)=1$ we have
$\nu(\partial I(\omega))=0$.
Moreover,
for every $q\geq M$, the construction gives
\[\nu_{q}(I(\omega))
=\frac{1}{(\#A_1^{\ell_1})^{t_1}\cdots
(\#A_m^{\ell_m})^{t_m}(\#A_{m+1}^{\ell_{m+1}})^{s}}\leq\frac{1}{
(\#A_m^{\ell_m})^{t_m}
(\#A_{m+1}^{\ell_{m+1}})^{s}}.\]
The weak* convergence of Borel probability measures on $[0,1]$ gives
\[\nu(I(\omega))=\lim_{q\to\infty}\nu_{q}(I(\omega))
\leq\frac{1}{
(\#A_m^{\ell_m})^{t_m}
(\#A_{m+1}^{\ell_{m+1}})^{s}}.\]
By \eqref{qe1} and
\eqref{covering}, for any $x\in\Lambda$ we have
\begin{equation}\label{eq-new}
\nu(B(x,r)\cap\Lambda)\leq 4
e^{-t_{m}\ell_m\left(h(\mu_{m})-\frac{1}{m}\right)}
e^{-s\ell_{m+1}\left(h(\mu_{m+1})-\frac{1}{m+1}\right)}.\end{equation}

If $0\leq s<t_{m+1}-1$ then
$M+1=t_1+\cdots+t_m+s+1$ by \eqref{k}, and
$r_{M+1}=e^{-\ell_{m+1}\left(\chi(\mu_{m+1})+ \frac{2}{m+1}\right)}r_M$
by \eqref{choose-eq0}.
From this and \eqref{eq-new},
\[
\begin{split}
\frac{\log\nu (B(x,r)\cap\Lambda)}{\log r_{M+1}}\geq&\frac{t_{m}\ell_{m}\left(h(\mu_{m})-
1/m\right)+s\ell_{m+1}\left(h(\mu_{m+1})-1/(m+1)\right)}
{t_{m}\ell_{m}\left(\chi(\mu_{m})+3/m\right)
+(s+1)\ell_{m+1}\left(\chi(\mu_{m+1})+2/(m+1)\right)}\\
&+O\left(\frac{1}{|\log r|}\right).\end{split}\]
Since $\liminf_{m\to\infty}h(\mu_m)>0$, 
the two terms in the numerator of the fraction are positive for any sufficiently large $m$.
If necessary, we replace $t_{m}$ by a sufficiently large integer 
so that
for $0\leq s<m$, the first term in the denominator
of the fraction dominates the second one
 so that
\[\frac{\log\nu (B(x,r)\cap\Lambda)}{\log r_{M+1}}\geq\frac{h(\mu_{m})-
1/m}
{\chi(\mu_{m})+4/m}+O\left(\frac{1}{|\log r|}\right).\]
If $m\leq s<t_{m+1}-1$, then 
we clearly have
\[\frac{\log\nu (B(x,r)\cap\Lambda)}{\log r_{M+1}}\geq\min\left\{\frac{h(\mu_{m})-
1/m}
{\chi(\mu_{m})+3/m},
\frac{m\left(h(\mu_{m+1})-1/(m+1)\right)}
{(m+1)\left(\chi(\mu_{m+1})+2/(m+1)\right)}\right\}+O\left(\frac{1}{|\log r|}\right).\]
In the remaining case $s=t_{m+1}-1$,
we have
 $M+1=t_1+\cdots+t_{m+1}$ 
 by \eqref{k} 
and $r_{M+1}=e^{-t_{m+1}\ell_{m+1}\left(\chi(\mu_{m+1})+ \frac{3}{m+1}\right)}$ by \eqref{choose-eq0}.
Using this and \eqref{eq-new},
\[\frac{\log\nu (B(x,r)\cap\Lambda)}{\log r_{M+1}}\geq\frac{
(t_{m+1}-1)\left(h(\mu_{m+1})-1/(m+1)\right)}
{t_{m+1}\left(\chi(\mu_{m+1})+2/(m+1)\right)}+O\left(\frac{1}{|\log r|}\right).\]
As $r\to0$
we have $M\to\infty$ and $r_M\to0$
by \eqref{ak}, and 
 \eqref{mass} follows. 
 This completes the proof of Theorem~\ref{main}. \qed
\subsection*{Acknowledgments}
I thank anonymous referees for their careful readings of the manuscript
and giving useful suggestions for improvements.
I thank Johannes Jaerisch for fruitful discussions.
This research was partially supported by the JSPS KAKENHI 
19K21835, 20H01811.

\end{document}